\documentclass[12pt, leqno]{amsart}
\usepackage[utf8]{inputenc}
\usepackage[english]{babel}
\usepackage{amstext}
\usepackage{amsthm}
\usepackage{amsmath}
\usepackage{amssymb}
\usepackage{graphicx}
\usepackage{mathtools}
\usepackage{xcolor} % Orange marks revisions for author review.
\usepackage{fancyvrb}
\usepackage[margin=1in]{geometry}

\usepackage[unicode=true,pdfusetitle,
 bookmarks=true,bookmarksnumbered=false,bookmarksopen=false,
 breaklinks=false,pdfborder={0 0 1},backref=false,colorlinks=true]
 {hyperref}
 \usepackage{mathrsfs}
 \usepackage{bm}

\usepackage[show]{ed}

\hypersetup{
 linkcolor=red,urlcolor=red,citecolor=red}

\makeatletter

\AtBeginDocument{}
\AtBeginDocument{}

\numberwithin{equation}{section}
\numberwithin{figure}{section}

\theoremstyle{plain}
\newtheorem{thm}{\protect\theoremname}[section]

\theoremstyle{definition}
\newtheorem{defn}[thm]{\protect\definitionname}

\theoremstyle{plain}
\newtheorem{lem}[thm]{\protect\lemmaname}

\theoremstyle{plain}

\theoremstyle{plain}
\newtheorem{prop}[thm]{\protect\propositionname}

\theoremstyle{plain}
\newtheorem*{claim*}{\protect\claimname}

\theoremstyle{remark}

\newtheorem{remark}[thm]{\protect\remarkname}

\theoremstyle{plain}
\newtheorem{assum}[thm]{Assumption}

\def\R{\mathbb R}

\def\eps{\varepsilon}

\makeatother

\providecommand{\definitionname}{Definition}
\providecommand{\lemmaname}{Lemma}
\providecommand{\propositionname}{Proposition}
\providecommand{\theoremname}{Theorem}
\providecommand{\remarkname}{Remark}
\providecommand{\claimname}{Claim}
\providecommand{\corname}{Corollary}

\theoremstyle{plain}

\title[NLS with inhomogeneous nonlinear damping and nonlinearity]{Scattering for Defocusing NLS with Inhomogeneous Nonlinear Damping and Nonlinear Trapping Potential}
\author[D. Lafontaine]{David Lafontaine}
\address{CNRS and Institut de Mathématiques de Toulouse ; UMR5219, Université de Toulouse ; CNRS, UPS IMT, F-31062 Toulouse Cedex 9 (France)}
\email{david.lafontaine@math.univ-toulouse.fr}

\author[B. Shakarov]{Boris Shakarov}
\address{Institut de Mathématiques de Toulouse ; UMR5219, Université de Toulouse ; CNRS, UPS IMT, F-31062 Toulouse Cedex 9 (France)}
\email{boris.shakarov@math.univ-toulouse.fr}
\thanks{This work was supported by the ANR LabEx CIMI (grant ANR-11-LABX-0040) within the French State
Program ``Investissements d’Avenir''. The work of B.\ S.\ is 
  partially supported by the CIMI, ANR-11-LABX-0040, and the 
  ANR project NQG ANR-23-CE40-0005}


\date{\today}

\begin{document}

\begin{abstract}
We investigate an energy-subcritical defocusing nonlinear Schr\"odinger equation in $\R^3$ subject to a lower order nonlinear trapping potential and a  spatially dependent nonlinear damping:
$$
    i\partial_t u + \Delta u + i a(x) |u|^{2\sigma_2} u
    = |u|^{2\sigma_1} u + V(x)\, |u|^{2\sigma_3} u.
$$
We prove that if the damping acts where $V$ induces concentration effects, i.e.\ where $V$ is either negative or non-repulsive, solutions are global and uniformly bounded in $H^1$, and scatter in the intercritical regime.
A primary challenge arises from the spatial dependence of $a(x)$, which breaks the energy's monotonicity. Consequently, a uniform in time control of the $H^1$ norm of a solution is non-trivial and represents a new result even for $V = 0$. We overcome this issue by introducing a novel \emph{energy modified by virial} argument, showing simultaneously a uniform bound on the energy and local energy decay estimates, which are subsequently upgraded to scattering via interaction Morawetz estimates.
\end{abstract}

\maketitle

\VerbatimFootnotes 

\section{Introduction}
We are interested in the following 
nonlinear Schrödinger equation with a spatially dependent nonlinear damping and a lower-order nonlinear potential
\begin{equation} \label{eqDNLS}
\begin{cases}
    i\partial_t u + \Delta u + i a(x) |u|^{2\sigma_2} u
    = |u|^{2\sigma_1} u + V(x)\, |u|^{2\sigma_3} u, \\
    u(0,x) = u_0(x) \in H^1(\R^3).
\end{cases}
\end{equation}
The exponents are assumed to be energy subcritical, satisfying $0< \sigma_1 <2,  0<\sigma_2\leq \sigma_1,$ and $0<\sigma_3<\sigma_1.
$
The term $i a |u|^{2\sigma_2} u$ represents a non-uniform in space nonlinear damping, with $a \in C^2(\R^3, \R^+)$. 
The first nonlinearity on the right-hand side is defocusing,
while the second involves a potential $V \in W^{1,\infty}(\R^3, \R)$, for which we make no a priori assumptions regarding its sign or repulsivity. Equation \eqref{eqDNLS} thus models the competition between dispersion, non-uniform in space nonlinear damping, and concentration effects induced by a nonlinear potential.  We aim to describe the resulting long-time dynamics. Specifically, we prove scattering forward in time to a linear solution provided that $a$ is active where the potential induces concentration, that is, where $V$ is either negative or non-repulsive.

It is well known that the following mass and energy functionals are conserved when $a=0$
\begin{align}\label{eqDefMass}
    M[u(t)] &:= \int |u(t,x)|^2 dx, \\ \label{eqDefEner}
    E[u(t)] &:=  \int \frac{1}{2} |\nabla u(t,x)|^2  + \frac{1}{2\sigma_1+2} |u(t,x)|^{2\sigma_1+2} + \frac{1}{2\sigma_3+2}  V(x)|u(t,x)|^{2\sigma_3+2} \; dx.
\end{align}

In the classical case $a=V=0$, the conservation of energy combined with the local well-posedness theory ensures global well-posedness in $H^1(\mathbb R^3)$. Moreover, it is known since the work of Ginibre and Velo \cite{GV85scat} %(see \cite{Naka} for the 1d and 2d cases)
that solutions scatter in $H^1$ to linear solutions, in the sense that there exists $u_\pm \in H^1(\mathbb R^3)$ so that
$$
\Vert u(t) - e^{it\Delta} u_\pm \Vert_{H^1(\mathbb R^3)} \to 0 \hspace{0.3cm}\text{ as }\hspace{0.3cm} t \to \pm\infty.
$$
When the undamped equation is perturbed, for instance by the addition of a potential as in our case, by a variable-coefficient Laplacian, or by the presence of an obstacle, a natural question arises: how does the geometry of the perturbation affect the scattering behavior?
For a \emph{linear} potential ($V\neq 0, \sigma_3=0$), the scattering result persists under the crucial assumption that $V$ is non-negative and repulsive (or has only small positive non-repulsive parts), namely $V \geq 0$ and $(\nabla V \cdot x)_+ \leq 0$, see e.g., \cite{BaVi16, FoVi18, Ho16, La16}. This condition can be interpreted as a strong non-trapping assumption, ensuring that all the Hamilton trajectories associated with the operator escape to infinity. Scattering to a linear solution has also been obtained in various other non-trapping settings, for instance, for equations posed outside convex or star-shaped obstacles with Dirichlet boundary conditions \cite{ Ab15, IvPl10, KiViXi16a, KiViXi16b, PV1, PV2, BiTeQi22}. 

By contrast, we consider a general setting where no geometric non-trapping assumptions are imposed on the perturbation. Instead, we investigate whether a damping term can mitigate concentration effects and restore scattering. In our previous work \cite{LaSh25}, we demonstrated that for a variable-coefficient Laplacian, a localized linear damping $(\sigma_2 = 0)$ active near the trapping region is sufficient to ensure scattering. The purpose of the present paper is to investigate an analogous question in the presence of a \emph{nonlinear} damping. We emphasize that nonlinear damping is significantly weaker than its linear counterpart. This can be illustrated, for instance, by the fact that a constant linear damping $(a(x)=a_0>0)$ induces exponential decay to zero, whereas for a constant nonlinear damping the decay is at most polynomial on bounded manifolds, and in $\R^n$ solutions may even fail to decay to zero asymptotically in time \cite{AnCaSp15}.
We emphasize that for $V=0$, the NLS with a constant nonlinear damping has been studied in \cite{AnCaSp15, AnSp10, KiMiSa25}, but there are no works, to our knowledge, about NLS with a space-dependent nonlinear damping. Such an inhomogeneous nonlinear damping introduces two significant 
hurdles: (i) the energy is no longer monotonic, rendering even a uniform bound in time on the $H^1$ norm of solutions
non-trivial, and (ii) since no lower bound on $a$ is assumed, no global decay on the solutions is enforced a priori.

It appears that, at least with the methods developed here, nonlinear damping seems too weak to counteract linear trapping perturbations, but it can mitigate the effects of the perturbation by a trapping \emph{nonlinear} potential. Recall that for the corresponding purely inhomogeneous nonlinear Schr\"odinger equation 
\begin{equation} \label{eq:pur_inh}
i\partial_t u + \Delta u =  V(x)\, |u|^{2\sigma_3} u,
\end{equation}
scattering is known for all data in the defocusing case ($V\geq 0$) provided $V$ satisfies the crucial repulsivity condition $\nabla V(x) \cdot x \leq 0$ \cite{Ba25Sc, BaMu25Sc, AlGrTa25Sc, CuLiZh26Dc}. In the focusing case of (\ref{eq:pur_inh}), the scale invariant potential $V(x):= - |x|^{-b}$ has been extensively studied, and scattering below a natural threshold is known in various settings 
\cite{CaCa22Sc, CaFaGuCa22Sc, CuLiZh25Arx,FaGu17Sc,GuMu21Sc,LiMiZh25Sc,MiMuZh21Sc,Mu22Sc}.

Our main results, Theorems \ref{th1} and \ref{thmDefoSc1}, establish the global existence and scattering of solutions to (\ref{eqDNLS}) for arbitrary $H^1$ data,  provided that the damping $a(x)$ is active where the trapping occurs, specifically where $V$ is either negative or non-repulsive. This is the content of the following assumption, where $x_+:= \max(x,0)$, $x_-:= \max(-x,0)$.
\begin{assum}\label{assumControl}(Control of the trapping).
    There exists a constant $c_0>0$ such that 
    $$
    V_- + (\nabla V \cdot x)_+ \leq c_0 a^{\frac{\sigma_3+1}{\sigma_2+1}}.
    $$
\end{assum}
For example, if the trapping regions $\operatorname{supp}(V_-)$ and $\operatorname{supp}(\nabla V \cdot x)_+$ are compact, this condition is verified as soon as 
$$
\operatorname{supp}(V_- +(\nabla V \cdot x)_+) \subset \{ a>0 \}.
$$
However, we do not make any compactness assumption on the trapping region. In this general setting, the exponent $\frac{\sigma_3+1}{\sigma_2+1}$ represents a balance between the two nonlinear effects. In the general case where $a$ and $V$ are not compactly supported, we impose additional decay conditions on $\Delta a$, on the trapping part of $V$, and on $\nabla V$:
\begin{assum}\label{assumA} (Decay of $\Delta a$) The damping function $a \in C^2(\R^3) \cap L^\infty(\R^3)$ satisfies 
    \begin{equation}
\label{eqDeltaACond}
|\Delta a| \lesssim
        \begin{cases}
             \langle  x \rangle^{-7(\sigma_2+1)} \ &\mbox{ if } \textcolor{black}{0<\sigma_2<\sigma_1}, \\
         \langle  x \rangle^{-1} \ &\mbox{ if } \sigma_2 = \sigma_1.
        \end{cases}
    \end{equation}
\end{assum}
\begin{assum}\label{assumV}(Decay of the trapping part) If $\sigma_2 \neq \sigma_3$, then
    $$
    V_- + (\nabla V \cdot x)_+ \lesssim 
    \begin{cases}
    \langle x \rangle^{- 7(\sigma_3+1)} & \text{ if }\sigma_3 < \sigma_2, \\
    \langle x \rangle^{-\frac{\sigma_3+1}{\sigma_1+1}} & \text{ if }\sigma_3 > \sigma_2. \\
    \end{cases}
    $$
\end{assum}
\begin{assum}\label{assumNV}(Decay of $\nabla V$) We have
    $$
    |\nabla V(x)| \lesssim 
    \langle x \rangle^{- 7(\sigma_3+1)}.
    $$
\end{assum}
Notably, if $\sigma_3 = \sigma_2$, Assumption \ref{assumV} is not necessary, and no decay on the trapping part is needed. In addition, Assumption \ref{assumNV} will be necessary only for our scattering result.

We recall that due to the variable-coefficient $a$, even when $V=0$, showing a uniform in time bound on the $H^1$ norm of solutions is challenging because the energy is not monotonic, due to the appearance of a sign-changing term in its derivative. Our first result addresses this by showing global well-posedness of solutions together with a uniform bound on their $H^1$-norm.
\begin{thm} \label{th1}
Assume that
$$
0< \sigma_1 <2, \quad 0<\sigma_2\leq \sigma_1, \quad 0<\sigma_3<\sigma_1,
$$
and that $0 \leq a\in C^2(\R^3) \cap L^\infty(\R^3)$ and $V\in W^{1,\infty}(\mathbb R^3) $ satisfy the control Assumption \ref{assumControl}, as well as Assumptions \ref{assumA} and \ref{assumV}. Then, for any $u_0 \in H^1(\R^3)$, there exists a unique  global solution $u \in C([0,\infty), H^1(\R^3))$ to \eqref{eqDNLS}. \textcolor{black}{Moreover there exists $C(u_0)>0$, such that} 
    \begin{equation*}
        \textcolor{black}{\Vert u \Vert_{L^\infty(\mathbb R^+,H^1(\mathbb R^3))}\leq C(u_0).}
    \end{equation*}
\end{thm}

It is well known that for scattering to occur, the nonlinearity must be sufficiently strong. Classical $H^1$-scattering results for the energy-subcritical NLS with data in $H^1$ typically hold in the intercritical range $\frac 23 < \sigma < 2$. However, if the nonlinearity is appropriately localized, this requirement can be relaxed. 
To capture this, we introduce the following definition:
\begin{defn}\label{assumLoc} (\textcolor{black}{Intercritical nonlinearities}) 
$$
    \text{\emph{Either} } \sigma > \frac 23, \quad \text{ \emph{or} } \quad \textcolor{black}{0<\sigma\leq\frac23\ \text{ and }\ }\psi, \nabla \psi \in L^{\frac{6}{4-5\sigma}}(\mathbb R^3).
$$
\end{defn}

We can now state our main scattering result.

\begin{thm}\label{thmDefoSc1}
    Suppose the assumptions of Theorem \ref{th1} hold. Assume further that Assumption \ref{assumNV} holds and one of the following conditions is satisfied: 
    \begin{enumerate}
        \item $\sigma_1 >2/3$ and both $(a, \sigma_2)$, $(V, \sigma_3)$ are intercritical,
        \item $1-a \in C_c(\R^3)$, with $\sigma_1,\sigma_2 \geq 2/3$, and \textcolor{black}{$(V, \sigma_3)$ is intercritical}. 
    \end{enumerate}
   Then, any solution  $u \in C([0,\infty), H^1(\R^3))$ to \eqref{eqDNLS} scatters in $H^{1}(\R^3)$, in the sense that there exists $u_+ \in H^1(\R^3)$ such that
   $$
   \Vert u(t) - e^{it\Delta}u_+ \Vert_{H^1(\R^3)} \to 0 \text{ as } t \to + \infty.
   $$
\end{thm}

We now provide an overview of the proof. Given the local well-posedness theory, global existence follows from establishing a uniform bound on the $H^1$ norm of the solution on its maximal interval of existence. However, even for $V \geq 0$, such a bound is not a priori given by the energy law, because the energy is not monotonic due to the spatial variations of $a$: indeed, for $V\geq 0$, we have formally
$$
\partial_t E[u(t)] = (\text{non-positive terms}) \textcolor{black}{+} \frac{1}{2\sigma_2 + 2} \int \Delta a |u|^{2\sigma_2 +2}.
$$
In the case of linear damping $(\sigma_2 = 0)$ studied in \cite{LaSh25}, the term involving $\Delta a$ was controlled via a Morawetz-type identity with a linear weight $\langle x \rangle$.
This identity generated a localized mass term $\int \int \langle x \rangle^{-7} |u|^{2} \, dx \, d\tau$, which bounds the problematic term, and, after controlling the non-signed terms coming from trapping via the control condition, one obtained global well-posedness and local energy decay via a convexity argument.

This strategy collapses when $\sigma_2 > 0$ for two reasons. First, the relevant $L^{2\sigma_2 + 2}$ norm no longer appears explicitly, and second, the terms involving $a$ in the Morawetz identity are no longer of lower order. To overcome this, we introduce a new \emph{energy modified by virial} argument. More precisely, for $V\geq 0$, we define
$$
\mathcal E[u(t)] := E[u(t)] - \eta I[u(t)],
$$
where $I$ is the Morawetz functional with weight $\langle x \rangle$, and $\eta > 0$ is a parameter. If $V \leq 0$, $E$ is replaced by a coercive modified energy $E_+$ above. As $I[u(t)]$ is bounded by $ E[u(t)]^{\frac 12}$ (resp. $ E_+[u(t)]^{\frac 12}$), it suffices to bound $\mathcal E[u(t)]$ to establish a uniform $H^1$ bound. 

We differentiate $\mathcal E[u(t)]$ in time and integrate back to study its evolution. For $\eta \gg 1$ sufficiently big, the problematic non-signed term $\int \int \Delta a |u|^{2\sigma_2 + 2}$ coming from the energy is absorbed by the virial. This $\eta \gg 1$ being fixed, the problematic terms in the virial are, in turn, after interpolation with the mass law $\int \int a |u|^{2\sigma_2+2}<\infty$, absorbed by the non-positive terms coming from the energy. Finally, the non-signed terms arising from the trapping are controlled via the control condition combined with the mass law. This argument yields both the uniform $H^1$ bound, thus global well-posedness, and, exploiting the positive terms in the virial, the local energy decay
$$
       \int \int \frac{|\nabla u|^2}{\langle x \rangle^3} + \frac{|u|^{2\sigma_1+2}}{\langle x \rangle} + \frac{|u|^2}{\langle x \rangle^7} \, dxd\tau< +\infty.
$$
With local energy decay and uniform $H^1$ bound established, we turn to scattering using an interaction-Morawetz argument, which yields the global space-time bound $\Vert u \Vert_{L^4(\mathbb R^3 \times \textcolor{black}{\mathbb R^+})} < + \infty$. Non-signed terms appear in the computation from the trapping, but due to the localization of the trapping, these are controlled using the previously obtained local energy decay. The last step consists in obtaining scattering exploiting the $L^4$ global space-time bound by bootstrapping to a global Strichartz bound. Finally, when $1-a$ is compactly supported, the local energy decay gives the supplementary information that $\Vert u \Vert_{L^{2\sigma_2+2}(\mathbb R^3 \times \textcolor{black}{\mathbb R^+})} < + \infty$, \textcolor{black}{which gives the required mass-critical spacetime bound directly when $\sigma_2=2/3$.}

We finish this introduction by giving several comments on our results and potential generalizations.

For $a=0$, the energy is conserved, and a uniform bound on the energy is straightforward (see Remark \ref{rk:gwp_a0}). However, we conjecture in this case, both by analogy with the case of a \emph{linear} potential $\sigma_3 = 0$ \cite{GuNaTs04, SoWe90}, and with the case of a nonlinear constant potential $\sigma_3 > 0$, $V(x) = V_0 <0$ \cite{BeLi83}, that if $\{ V < 0\} \neq \emptyset$, stationary states exist for $(\ref{eqDNLS})$ with $a=0$ and thus scattering does not hold in general.  

As previously noted, nonlinear damping appears insufficient to counteract the trapping effects of a linear potential ($\sigma_3 = 0$) (which is stronger than a nonlinear perturbation where the solution is small), and our result holds only for nonlinear perturbations $\sigma_3>0$. Whether this limitation is intrinsic or a technical artifact of our method %virial-based estimates 
remains an open question. Obtaining a counter-example in the case $\sigma_3 = 0$, or extending our result down to $\sigma_3  = 0$, at least for some range of $\sigma_2 > 0$, seems a difficult open question. A similar question holds for metric perturbations or perturbations by an obstacle, in the spirit of \cite{LaSh25}. 

Finally,  a natural extension of this work involves the case where the primary nonlinearity $\sigma_1$ is \emph{focusing}. In this case, recall that for the corresponding equation with $a=0$ and $V=0$, scattering holds below a threshold given by the ground state associated with the equation, as shown by \cite{HR, DHR, FXC} after the breakthrough of Kenig and Merle \cite{KM} in the critical case (see also \cite{DMsimple, MurphyNR}). 
%Therefore, a tentative generalization of our result would be scattering for the focusing counterpart of (\ref{eqDNLS}) under the control control, for data under a natural energy threshold, 
A natural generalization of our results to the focusing case 
 would therefore involve identifying a scattering criterion related to the ground state of the unperturbed equation, for the focusing counterpart of (\ref{eqDNLS}) under the control condition. However, the non-monotonicity of the energy makes both identifying such a criterion and adapting the aforementioned methods particularly challenging.
In addition, this would likely involve profile decompositions in the spirit of \cite{KM}, which seems difficult to carry out here, particularly because the evolution is not a priori bounded in negative times.

Finally, still in the focusing case, the question of whether or not a constant damping $a(x) = a_0 > 0$ can prevent blow up has attracted significant attention recently, see \cite{AnSh25, SulemSulem99} and references therein. A natural question, in the spirit of our paper, is whether or not a non-constant, localized damping can prevent blow-up locally where it is active.

\subsection*{Notations}
We write $a \lesssim b$ to indicate that there exists a universal constant $C>0$ such that $a \leq Cb$. We use the notation $\langle x \rangle := \sqrt{1 + |x|^2}$.

\section{Preliminaries}
For any $f,g \in L^2(\R^3)$, we use the inner product defined by 
\begin{equation*}
    \langle f,g\rangle = Re \int f \bar{g} dx.
\end{equation*}
\subsection{Strichartz Estimates and local well-posedness}
A couple $(p,q)$ is said to be admissible if
\begin{equation}\label{eq:StrichS}
2\leq p, q \leq \infty , \hspace{0.3cm} \frac 2 p + \frac 3 q = \frac 3 2 .
\end{equation}
We recall the following Strichartz estimates for the (undamped) Schr\"odinger flow, due, among other works, to \cite{St, GV85scat, GV0, LS, KT}.
\begin{prop} \label{prop:strich}
For every interval $I=(t_0,T)$, $s\geq0$, and admissible pairs
$(p,q)$ and $(\widetilde p,\widetilde q)$, there is a constant independent of $I$ such that
\begin{equation}\label{eq:StrLossHom}
\|e^{i(t-t_0)\Delta}u_0\|_{L^p(I,W^{s,q})}\leq C\|u_0\|_{H^s},
\end{equation}
\textcolor{black}{
\begin{equation}\label{eq:NonHomDG}
\left\|\int_{t_0}^t e^{i(t-\tau)\Delta}f(\tau)\,d\tau\right\|_{L^p(I,W^{s,q})}
\leq C\|f\|_{L^{\widetilde p'}(I,W^{s,\widetilde q'})}.
\end{equation}}
\textcolor{black}{The dual homogeneous estimate also gives
\begin{equation}\label{eq:DualStrichartz}
\left\|\int_I e^{-i\tau\Delta}f(\tau)\,d\tau\right\|_{H^s}
\leq C\|f\|_{L^{\widetilde p'}(I,W^{s,\widetilde q'})}.
\end{equation}}
\end{prop}
\begin{remark}
    \textcolor{black}{In particular, the spaces $L^1H^s$, $L^{10/7}W^{s,10/7}$,
and $L^2W^{s,6/5}$ are allowed, by taking $(\widetilde p,\widetilde q)$ equal to
$(\infty,2)$, $(10/3,10/3)$, and the end point $(2,6)$, respectively.}
\end{remark}

As a classical consequence, the following local well-posedness result follows.

\begin{prop}\label{prpLocExsCubic} 
    Assume that $a, V\in W^{1,\infty}(\R^3)$ and $0 \leq \sigma_1,\sigma_2, \sigma_3 <2$.
For any $u_0 \in H^{1} (\R^3)$ there exists $T>0$ and a unique solution  $u \in C([0,T],H^{1}(\R^3))$ to (\ref{eqDNLS}). 
In addition, (i) if $\Vert u_0 \Vert_{H^{1}(\R^3)}$ is bounded,  then $T$ is bounded below, and (ii) the map $u_0 \in H^{1}(\R^3) \mapsto u \in C([0,T],H^{1}(\R^3))$ is continuous.
    In particular, if $T_{\rm max} >0$ denotes the maximal forward time of existence, either $T_{\rm max} = +\infty$ or
$$
        \lim_{t \to T_{\rm max}^-}\|  u (t) \|_{H^{{1}}(\R^3)} = \infty.
$$
\end{prop}
\begin{proof}
As $a,V\in W^{1,\infty}(\R^3)$, the result  
is classical, see e.g.\ \cite{Ca03}.
\end{proof}
\textcolor{black}{Under Assumption \ref{assumA}, $a$ and $\Delta a$ are bounded. Interior elliptic estimates on translated unit balls imply $\|\nabla a\|_{L^\infty}\lesssim\|a\|_{L^\infty}+\|\Delta a\|_{L^\infty}$, so the coefficient regularity required by Proposition \ref{prpLocExsCubic} and the later Strichartz estimates holds.}

\subsection{Evolution of the mass, the energy, and the Morawetz functional}
We compute the evolution in time of the mass and the energy defined respectively in \eqref{eqDefMass} and \eqref{eqDefEner}. 
Observe that in the following proofs of Propositions \ref{prop:NRJ} and \ref{prop:Mor}, as well as Propositions \ref{propGWP} and \ref{prop:L4L4}, we can assume by a standard approximation argument that $u$ is sufficiently regular and decaying so that all the formal computations are justified.

\begin{prop} \label{prop:NRJ}
Let  $u \in C([0,T_{\rm max}), H^1(\R^3))$ be a solution to \eqref{eqDNLS}. Then for any $0\leq t_0 \leq t < T_{\rm max}$, we have 
\begin{equation}\label{eq:MassEvol}
 M[u(t)] = M[u(t_0)] - 2\int_{t_0}^t \int a |u|^{2\sigma_2 + 2} dx ds.
\end{equation}
Moreover, if $u$ is sufficiently regular, then
\begin{equation}\label{eq:EnerEvol1}
    \begin{aligned}
    & E[u(t)] - E[u(t_0)] \\
    & = - \int_{t_0}^t \int a \big( |u|^{2\sigma_1 + 2\sigma_2 +2} + V|u|^{2\sigma_3 + 2\sigma_2 +2} + |u|^{2\sigma_2} | \nabla u |^2 + 2\sigma_2  |u|^{2\sigma_2} |\nabla|u||^2 \big) \,dx \,d\tau\\ 
     & \textcolor{black}{+\frac{1}{2\sigma_2+2}\int_{t_0}^t} \int \Delta a |u|^{2\sigma_2 + 2} dx \,d\tau.
\end{aligned}
\end{equation}
\end{prop}

\begin{proof} 
To obtain  \eqref{eq:EnerEvol1}, we take the duality product of \eqref{eqDNLS} with $\partial_t u$, and we obtain 
    \begin{align*}
        0 = \langle \partial_t u, i \partial_t u\rangle = \langle \partial_t u, - \Delta u + |u|^{2\sigma_1} u - i a |u|^{2\sigma_2} u\rangle =  \frac{d}{dt} E[u(t)] - \langle \partial_t u, i a |u|^{2\sigma_2} u\rangle.
    \end{align*}
    As 
    \begin{align*}
        - \langle \partial_t u, i a |u|^{2\sigma_2} u\rangle & = - \langle i \Delta  u - i|u|^{2\sigma_1} u - a|u|^{2\sigma_2} u, i a|u|^{2\sigma_2} u\rangle \\
        &= \langle |u|^{2\sigma_1} u, a|u|^{2\sigma_2} u\rangle + \langle \nabla u , a \nabla |u|^{2\sigma_2} u\rangle + \langle \nabla u, |u|^{2\sigma_2} u \nabla a \rangle,
    \end{align*}
    \eqref{eq:EnerEvol1} follows by integration by parts and time integration. Similarly, by taking the duality product of \eqref{eqDNLS} with $iu$, we get 
    \begin{equation*}
       \langle i\partial_t u, i u\rangle = \frac{1}{2} \frac{d}{dt} \| u \|_{L^2(\R^3)}^2  = - \langle a |u|^{2\sigma_2} u,u\rangle
    \end{equation*}
    and \eqref{eq:MassEvol} follows after integration in time.
\end{proof}

\begin{prop} \label{prop:Mor}
Let $\chi \in C^4(\mathbb R^3)$. Let \textcolor{black}{$u\in C([0,T_{\rm max}),H^1(\R^3))$} be a solution to \eqref{eqDNLS} and
\begin{equation*}
     I[u(t)] = \operatorname{Im} \int  \bar u(t) \nabla  u(t) \cdot \nabla \chi dx.
\end{equation*}
Then the following identity holds for $0\leq t_0\leq t<T_{\rm max}$:
    \begin{equation}\label{eqVirial}
    \begin{aligned}
   & I[u(t)] - I[u(t_0)] \\
   &= \int_{t_0}^t \int  2 D^2 \chi \nabla u \cdot \nabla \bar u - \frac 12 \Delta^2 \chi |u|^2+ \frac {\sigma_1}{\sigma_1+1}  \Delta \chi |u|^{2\sigma_1 + 2} + \frac {\sigma_3}{\sigma_3+1}  \Delta \chi V \textcolor{black}{|u|^{2\sigma_3+2}} 
     \\ & \quad -\frac{1}{\sigma_3 +1}  |u|^{2\sigma_3 + 2} \nabla \chi \cdot \nabla V - 2 a |u|^{2\sigma_2} \nabla \chi \cdot \operatorname{Im} ( \bar u \nabla u  ) \, dx d\tau.
      \end{aligned}
    \end{equation}
\end{prop}

\begin{proof}
Observe that
\begin{equation*}
    \frac{d}{dt} I[u(t)] = - \int \operatorname{Im} \left( \bar u \partial_t u \right) \Delta \chi \, dx - 2 \int \operatorname{Im} \left( \partial_t u \nabla \bar u \right) \cdot \nabla \chi \, dx = {\rm A} + {\rm B}.
\end{equation*}
By \eqref{eqDNLS} we have
\begin{equation*}
\begin{aligned}
      {\rm A} &= \int (|u|^{2\sigma_1 + 2} + |\nabla u |^2 + V|u|^{2\sigma_3 + 2})\Delta \chi -\frac{1}{2} |u|^2 \Delta^2 \chi  \, dx,
\end{aligned}
\end{equation*}
and 
\begin{equation*}
    \begin{aligned}
         {\rm B} &= \int 2  D^2 \chi \nabla u  \cdot \nabla \bar u - \left(\frac{1}{\sigma_1 + 1} |u|^{2\sigma_1 + 2} + |\nabla u|^2 + \frac{V}{\sigma_3+1}|u|^{2\sigma_3 + 2} \right) \Delta \chi \, dx \\
         &\quad -\frac{1}{\sigma_3 +1} \int |u|^{2\sigma_3 + 2} \nabla \chi \cdot \nabla V \,dx + 2\int a |u|^{2\sigma_2} \operatorname{Im}(u \nabla \bar u) \cdot \nabla \chi dx.
    \end{aligned}
\end{equation*}
We get the result by summing the two equations above.
\end{proof}

\section{Uniform bound on the energy and local energy decay} \label{S:UBE}
The main purpose of this section is to show a uniform-in-time bound in the $H^1$-norm of the solution, as well as the local energy decay in \eqref{eqLocEnDec}. Notice that, combined with the local well-posedness theory, these estimates imply Theorem \ref{th1}. 
\begin{prop}\label{propGWP} 

\textcolor{black}{Under the hypothesis of Theorem \ref{th1}}, there exists a constant $K = K(\| u_0\|_{H^1})>0$, such that every solution $u$ to \eqref{eqDNLS} on $[0,T]$ for any $T >0$ satisfies
\begin{equation}\label{eqEnUnBn}
\sup_{t\in[0,T]}\|u(t)\|_{H^1(\R^3)}^2\leq  K(u_0),
\end{equation}
\begin{equation}\label{eqEnUnBnRE}
\sup_{t\in[0,T]}|E[u(t)]|\leq  K(u_0),
\end{equation}
\begin{equation}\label{eqLocEnDec}
\int_0^T\int\left(\frac{|\nabla u|^2}{\langle x\rangle^3}+\frac{|u|^{2\sigma_1+2}}{\langle x\rangle}+\frac{|u|^2}{\langle x\rangle^7}\right)\,dx\,d\tau\leq K(u_0),
\end{equation}
\begin{equation}\label{eqAIntTerm}
\int_0^T\int a|u|^{2\sigma_2}\bigl(|u|^{2\sigma_1+2}+|\nabla u|^2+|u|^2\bigr)\,dx\,d\tau\leq  K(u_0).
\end{equation}
\end{prop}

\begin{proof}
As a preliminary, we begin by defining a first coercive positive energy. We recall the following interpolation inequality, which is a simple consequence of the Young inequality for products 
$ab \leq \frac 1p a^p + \frac 1q b^q$ for any $a,b\geq0$ and $\frac 1p + \frac 1q =1$, which we will use several times in this work:
\begin{equation} \label{eq:Y_int} 
\begin{gathered}
\forall 1 \leq p < r < q < \infty, \quad \forall \epsilon >0, \quad \exists C_{\eps} >0 \quad \text{s.t.}\\
\forall x \geq 0, \quad x^r \leq C_\eps x^p + \epsilon x^q \quad \text{ and } \quad x^r \leq \eps x^p + C_\epsilon x^q.
\end{gathered}
\end{equation}
By (\ref{eq:Y_int}), there exists $\Lambda > 0$ so that 
\begin{equation} \label{eq:Lambda}
\forall x \geq 0, \quad  \Vert V \Vert_{L^\infty}x^{2\sigma_3 + 2} \leq  \Lambda  x^2 + \frac{1}{2\sigma_1+2}   x^{2\sigma_1 + 2}.
\end{equation}
This $\Lambda > 0$ being fixed, we define
\begin{equation} \label{eq:defposE}
E_+[u(t)] :=  E[u(t)] + \Lambda \int |u(t)|^2 dx,
\end{equation}
so that, by (\ref{eq:Lambda}),
$$
E_+[u(t)] \geq \frac 12 \int |\nabla u(t)|^2 dx.
$$
We will now define a \emph{modified energy-virial functional}. To this end, let $\chi$ be the weight function 
\begin{equation*}
    \chi(x) := \langle x \rangle =  \sqrt{|x|^2 + 1},
\end{equation*}
and $I$ be the associated Morawetz functional, as in Proposition \ref{prop:Mor}.
For future reference, notice that
\begin{equation}\label{eqChiComp1}
   \Delta \chi = \frac{1}{\langle x \rangle}+ \frac{2}{\langle x \rangle^3}, \quad  -\Delta^2 \chi  = \frac{15}{\langle x \rangle^7}, \quad D^2 \chi \xi \cdot \xi \geq \frac{|\xi|^2}{\langle x \rangle^3}.
\end{equation}
For some $\eta>0$ to be fixed later, we define the modified energy-virial functional
\begin{equation} \label{ref:def_modE}
    \mathcal{E}[u(t)] := E_+[u(t)] - \eta I[u(t)].
\end{equation}
By the energy law \eqref{eq:EnerEvol1}, the mass law \eqref{eq:MassEvol} and the Morawetz identity \eqref{eqVirial}, we get 
\[
\begin{aligned}
\frac{d}{dt}\mathcal E =&-\int a\Bigl(|u|^{2\sigma_1+2\sigma_2+2}
+V|u|^{2\sigma_3+2\sigma_2+2}
+|u|^{2\sigma_2}|\nabla u|^2
+2\sigma_2|u|^{2\sigma_2}|\nabla|u||^2\Bigr)\,dx\\
&\textcolor{black}{+}\frac{1}{2\sigma_2+2}\int\Delta a\,|u|^{2\sigma_2+2}\,dx
-\textcolor{black}{2}\Lambda\int a|u|^{2\sigma_2+2}\,dx\\
&-\eta\int\Bigl(2D^2\chi\nabla u\cdot\nabla\bar u
-\frac12\Delta^2\chi|u|^2
+\frac{\sigma_1}{\sigma_1+1}\Delta\chi|u|^{2\sigma_1+2}\Bigr)\,dx\\
&-\eta\frac{\sigma_3}{\sigma_3+1}\int\Delta\chi V|u|^{2\sigma_3+2}\,dx
+\frac{\eta}{\sigma_3+1}\int|u|^{2\sigma_3+2}\nabla\chi\cdot\nabla V\,dx\\
&+2\eta\int a|u|^{2\sigma_2}\nabla\chi\cdot\operatorname{Im}(\bar u\nabla u)\,dx.
\end{aligned}
\]
Using \eqref{eq:Lambda}, we get
\[
\int aV_-|u|^{2\sigma_3+2\sigma_2+2}\,dx
\leq\Lambda\int a|u|^{2\sigma_2+2}\,dx
+\frac{1}{2\sigma_1+2}\int a|u|^{2\sigma_1+2\sigma_2+2}\,dx.
\]
\textcolor{black}{Since $1-(2\sigma_1+2)^{-1}\geq1/2$}, we discard the remaining nonpositive terms to obtain
\begin{equation}\label{eqDerModEn1}
\begin{aligned}
\frac{d}{dt}\mathcal E\leq{}&-\int a|u|^{2\sigma_2}\Bigl(\textcolor{black}{\frac12}|u|^{2\sigma_1+2}+|\nabla u|^2\Bigr)\,dx
\textcolor{black}{+}\frac{1}{2\sigma_2+2}\int\Delta a\,|u|^{2\sigma_2+2}\,dx\\
&-\eta\int\Bigl(2D^2\chi\nabla u\cdot\nabla\bar u
-\frac12\Delta^2\chi|u|^2
+\frac{\sigma_1}{\sigma_1+1}\Delta\chi|u|^{2\sigma_1+2}\Bigr)\,dx\\
&+\frac{\eta}{\sigma_3+1}\int\Bigl((\nabla\chi\cdot\nabla V)_+
+\sigma_3\Delta\chi V_-\Bigr)|u|^{2\sigma_3+2}\,dx\\
&+2\eta\int a|u|^{2\sigma_2}\nabla\chi\cdot\operatorname{Im}(\bar u\nabla u)\,dx.
\end{aligned}
\end{equation}

We will use the positive terms coming from the virial to absorb the non-signed term coming from the derivative of the energy for $\eta>0$ big enough. This $\eta>0$ being fixed, the non-signed term in the virial involving $a$ will be interpolated between the mass law and the positive terms coming from the energy, while the non-signed terms in the virial originating from the trapping will in turn be absorbed thanks to the control condition \ref{assumV}, leading both to the uniform $H^1$ bound and the local energy decay.

We next absorb the $\Delta a$ term. If $0<\sigma_2<\sigma_1$, \eqref{eq:Y_int} gives
\[
\left|\int\Delta a\,|u|^{2\sigma_2+2}\,dx\right|
\lesssim\int\Bigl(|\Delta a|^{\frac{1}{\sigma_2+1}}|u|^2
+|\Delta a|^{\frac{\sigma_1+1}{\sigma_2+1}}|u|^{2\sigma_1+2}\Bigr)\,dx.
\]
Assumption \ref{assumA} and \eqref{eqChiComp1} bound the two weights by constants times $-\Delta^2\chi$ and $\Delta\chi$, 
\begin{equation*}
   \left| \int \Delta a |u|^{2\sigma_2 + 2} dx \right| \lesssim \int|\Delta a|^\frac{1}{\sigma_2 + 1}|u|^2 dx + \int |\Delta a|^\frac{\sigma_1+1}{\sigma_2+1} |u|^{2\sigma_1 + 2} dx,
\end{equation*}
When $\sigma_2=\sigma_1$, only the second term is needed, and $|\Delta a|\lesssim\Delta\chi$. Set $$
\textcolor{black}{\kappa:=\min\left\{1/4,\sigma_1/(2(\sigma_1+1))\right\}>0.}$$ Observe that by \eqref{eqChiComp1} and the decay assumption \eqref{eqDeltaACond} on $\Delta a$,
\begin{equation*}
     |\Delta a|^\frac{1}{\sigma_2+1} \lesssim -\Delta^2 \chi,\quad \forall 0<\sigma_2<\sigma_1; \quad\text{ and }\quad |\Delta a|^\frac{\sigma_1 +1}{\sigma_2+1} \lesssim \Delta \chi, \quad \forall 0<\sigma_2\leq\sigma_1.
\end{equation*}
Thus, 
there exists $\eta>0$ big enough such that 
\begin{equation}\label{eqDerModEn2}
\begin{aligned}
      & - \frac{1}{2\sigma_2 +2} \int \Delta a |u|^{2\sigma_2 + 2} dx - \eta \int - \frac 12 \Delta^2 \chi |u|^2+ \frac {\sigma_1}{\sigma_1+1}  \Delta \chi |u|^{2\sigma_1 + 2}dx \\ 
      &\leq - \textcolor{black}{\kappa} \eta \int - \Delta^2 \chi |u|^2+ \Delta \chi |u|^{2\sigma_1 + 2}dx \leq 0.
\end{aligned}
\end{equation}
We now fix such an $\eta > 0$. By (\ref{eqDerModEn1}) and  (\ref{eqDerModEn2}), we get in particular

 \begin{equation}
 \begin{aligned} \label{eqDerModEn}
    &\mathcal{E}[u(t)] -  \mathcal{E}[u(0)]   \leq   
    - \int_0^t \int a |u|^{2\sigma_2} \Big( \frac{2\sigma_3 + 1}{2\sigma_3+2}|u|^{2\sigma_1 +2}  +| \nabla u |^2 \Big)dx \, d\tau \\
    & \hspace{1cm}- \int_0^t \int a  |u|^{2\sigma_2} | \nabla u |^2 \,dx \, d\tau  + 2\eta \int_0^t \int  a |u|^{2\sigma_2} \nabla \chi \cdot \operatorname{Im} ( \bar u \nabla u) \,dx \, d\tau \nonumber\\
      & \hspace{1cm} - \eta\int_0^t\int D^2\chi \nabla u \cdot \nabla \bar u + \textcolor{black}{\kappa} \Delta \chi |u|^{2\sigma_1 + 2} -  \textcolor{black}{\kappa} \Delta^2 \chi |u|^2  \,dx \, d\tau\\
      & \hspace{1cm} + \eta \int_0^t \int \frac{1}{\sigma_3 +1}  |u|^{2\sigma_3 + 2} (\nabla \chi \cdot \nabla V)_+ + \frac {\sigma_3}{\sigma_3+1}  \Delta \chi V_- |u|^{2\sigma_3 + 2}dx \, d\tau.
\end{aligned}
 \end{equation}
Observe that by Cauchy-Schwarz and Young inequality for products, for any $\eps>0$ there exists $C_\eps >0$ such that
\begin{equation}\label{eqBadTerm}
    \begin{aligned}
    & \Big|\eta \int_0^t \int  a |u|^{2\sigma_2}  \nabla \chi \cdot \operatorname{Im} ( \bar u \nabla u  ) dx \, d\tau  \Big|  \\
    & \leq \eta \left( \int_0^t \int  a| \nabla \chi| |u|^{2\sigma_2}  | \nabla u |^2 dx \, d\tau \right)^\frac{1}{2} \left( \int_0^t \int  a| \nabla \chi| |u|^{2\sigma_2+2} dx \, d\tau \right)^\frac{1}{2} dx \, d\tau \\
    & \leq C_0 \eta \eps  \int_0^t \int  a|u|^{2\sigma_2}  | \nabla u |^2  + \eta C_\eps  \int_0^t \int  a |u|^{2\sigma_2+2} dx \, d\tau. 
\end{aligned}
\end{equation}
for some constant $C_0>0$ depending only on $\chi$, where we used that $\nabla \chi$ is bounded.

We now handle the trapping part. Let $\delta >0$ be fixed small enough later. By interpolation--Young inequality (\ref{eq:Y_int}), if $2\sigma_3+2 \in (2,2\sigma_2+2]$, there exists $C_\delta > 0$ so that
\begin{align*}
\int_0^t\int   |u|^{2\sigma_3 + 2} (\nabla \chi \cdot \nabla V)_+   dx \, d\tau  
& \leq \delta \int_0^t\int (\nabla \chi \cdot \nabla V)_+ ^{\frac{2}{2\sigma_3 + 2}} |u|^{2} dx \, d\tau \\
& \quad + C_\delta \int_0^t\int (\nabla \chi \cdot \nabla V)_+^{\frac{2\sigma_2+2}{2\sigma_3 + 2}}|u|^{2\sigma_2 + 2} dx \, d\tau .
\end{align*}
We use the decay Assumption \ref{assumV} on the first term, and the control Assumption \ref{assumControl} on the second term to get
$$
\int_0^t\int   |u|^{2\sigma_3 + 2} (\nabla \chi \cdot \nabla V)_+ dx \, d\tau \leq \delta C_1 \int_0^t\int (-\Delta^2 \chi) |u|^{2} dx \, d\tau + C_\delta C_2 \int_0^t\int a |u|^{2\sigma_2 + 2} dx \, d\tau,
$$
for some $C_1, C_2 > 0$.
Similarly, if $2\sigma_3+2 \in [2\sigma_2+2,2\sigma_1+2)$, we get
$$
\int_0^t\int   |u|^{2\sigma_3 + 2} (\nabla \chi \cdot \nabla V)_+ dx \, d\tau  \leq \delta C_3 \int_0^t\int \Delta \chi |u|^{2\sigma_1 + 2} dx \, d\tau + C'_\delta C_4  \int_0^t\int a |u|^{2\sigma_2 + 2} dx \, d\tau.
$$
Thus, in both cases, setting $C''_\delta := C_\delta C_2 + C'_\delta C_4$, it follows that
\begin{equation} \label{eq:trapp_term}
\begin{aligned}
    \int_0^t\int   |u|^{2\sigma_3 + 2} (\nabla \chi \cdot \nabla V)_+ dx \, d\tau & \leq 
\delta C_1 \int_0^t\int (-\Delta^2 \chi) |u|^{2} dx \, d\tau \\
& \quad + \delta C_3 \int_0^t\int \Delta \chi |u|^{2\sigma_1 + 2} dx \, d\tau + C''_\delta \int_0^t\int a |u|^{2\sigma_2 + 2} dx \, d\tau.
\end{aligned}
\end{equation}
Similarly, we get
\begin{equation} \label{eq:trapp_term2}
\begin{aligned}
    \int_0^t\int   \Delta \chi V_- |u|^{2\sigma_3 + 2} dx \, d\tau  & \leq 
\delta C_5 \int_0^t\int (-\Delta^2 \chi) |u|^{2} dx \, d\tau \\
&\quad + \delta C_6 \int_0^t\int \Delta \chi |u|^{2\sigma_1 + 2} dx \, d\tau + C'''_\delta \int_0^t\int a |u|^{2\sigma_2 + 2} dx \, d\tau.
\end{aligned}
\end{equation}
We now fix 
$$
\delta := \frac {\sigma_3+1}{8\max(\sigma_3, 1)\max(C_1 + C_5, C_3+C_6)}, \quad \epsilon := \frac{1}{2C_0\eta},
$$

and we inject \eqref{eqBadTerm} and \eqref{eq:trapp_term}, \eqref{eq:trapp_term2}  back into (\ref{eqDerModEn})
to get
\begin{equation}\label{eq:Eboundfin}
    \begin{aligned}
    \mathcal{E}[u(t)] - \mathcal{E}[u(0)] & \leq  
    - \int_0^t \int a |u|^{2\sigma_2} \Big( \frac{2\sigma_3 + 1}{2\sigma_3+2}|u|^{2\sigma_1 +2}  +| \nabla u |^2 dx \, d\tau \Big) \\
    &\quad-  \eta \int_0^t \int 2 D^2 \chi \nabla u \cdot \nabla \bar u  - \frac{\kappa}{2}\Delta^2 \chi |u|^2+ \frac{\kappa}{2}\Delta \chi |u|^{2\sigma_1 + 2} dx \, d\tau 
     \\
     &\quad + (\eta C_\epsilon + C''_\delta + C'''_\delta) M[u(0)],
     \end{aligned}
\end{equation}
where we used the mass law \eqref{eq:MassEvol} to control
$$
\int_0^t\int a |u|^{2\sigma_2 + 2} dx \, d\tau \leq M[u(0)].
$$

Consequently, as $D^2 \chi$ is positive definite and $-\Delta^2 \chi \geq 0$, $\Delta \chi \geq 0$, from (\ref{eq:Eboundfin}) we get the bound from above
\begin{equation} \label{ref:modBound}
     \mathcal{E}[u(t)] \leq   \mathcal{E}[u(0)] + (\eta C_\epsilon + C_\delta '' + C_\delta ''')M[u(0)]. 
\end{equation}
\end{proof}

\begin{remark}\label{rk:gwp_a0}
For $a=0$, conservation of the coercive energy $E_+$ still gives global well-posedness. If $V_-+(\nabla V\cdot x)_+\not\equiv0$, however, the control Assumption \ref{assumControl} fails, and the preceding proof does not yield local energy decay. As discussed in the introduction, we conjecture that stationary states may obstruct scattering when $V_-\not\equiv0$.
\end{remark}

\section{Bilinear Estimates} \label{S:IM}
The purpose of this section is to show the following 
\begin{prop} \label{prop:L4L4}
    \textcolor{black}{Assume all hypotheses of Theorem \ref{th1}, together with Assumption \ref{assumNV}. Then for any solution $u$ to \eqref{eqDNLS} there exists $C(u_0)>0$, such that} 
    \begin{equation}\label{eq:GlStrEst}
        \int_0^\infty \int |u(x,\textcolor{black}{\tau})|^4 dx \, d\tau \lesssim C(u_0).
    \end{equation}
\end{prop}

\begin{proof}
    Define $\rho(x) := |x|$, and
\begin{equation*}
    B[u(t)] := \int |u(y)|^2 \int \operatorname{Im}(\bar u(x) \nabla u(x)) \cdot \nabla \rho(x-y)dx dy.
\end{equation*}
For future reference, note that
\begin{equation}\label{eqRho}
    \nabla \rho = \frac{x}{|x|}, \hspace{0.3cm} \Delta \rho = \frac{2}{|x|}, \hspace{0.3cm} \nabla \Delta \rho = \textcolor{black}{-\frac{2x}{|x|^3}}, \hspace{0.3cm} \Delta^2 \rho = -8\pi \delta_{x=0}.
\end{equation}
Observe already that, by the Cauchy-Schwarz inequality, the uniform bound on the energy given by Proposition \ref{propGWP}
$$
|B[u(t)]| \leq \Vert u \Vert^3_{L^2} \Vert \nabla u \Vert_{L^2} \lesssim 1.
$$
On the other hand, we have
\begin{equation}\label{eq:Bil3Terms}
    \begin{aligned}
        \frac{d}{dt} B[u(t)] &= 2 \iint  \operatorname{Re}(\bar u(y) \partial_t u(y)) \operatorname{Im}(\bar u(x) \nabla u(x)) \cdot \nabla \rho(x-y) dx dy \\ 
        &\quad+ \iint |u(y)|^2 \operatorname{Im}(\partial_t \bar u(x) \nabla u (x)) \cdot \nabla \rho(x-y) dx dy \\
        &\quad + \iint |u(y)|^2 \operatorname{Im}( \bar u(x) \nabla \partial_t u (x)) \cdot \nabla \rho(x-y) dx dy \\
        & = 2 \iint \operatorname{Re}(\bar u(y) \partial_t u(y)) \operatorname{Im}(\bar u(x) \nabla u(x)) \cdot \nabla \rho(x-y) dx dy \\
        &\quad - \iint |u(y)|^2 \operatorname{Im}(\bar u(x) \partial_t u(x)) \Delta \rho(x-y) dx dy \\
        &\quad - 2 \iint |u(y) |^2 \operatorname{Im}(\partial_t u(x) \nabla \bar u(x)) \cdot \nabla \rho(x-y) dx dy \\
        &=: {\rm I} + {\rm II} + {\rm III}.
    \end{aligned}
\end{equation}
By \eqref{eqDNLS}, we have 
\begin{equation*}
    \begin{aligned}
         & \frac{1}{2} {\rm I} \\
         &= \iint \operatorname{Re}\bigl(i\bar u(y)\Delta u(y)-a(y)|u(y)|^{2\sigma_2+2}\bigr)
\operatorname{Im}(\bar u(x)\nabla u(x))\cdot\nabla\rho(x-y)\,dx\,dy \\
        &= - \iint a(y) |u(y)|^{2\sigma_2 +2} \operatorname{Im}(\bar u(x) \nabla u(x)) \cdot \nabla \rho(x-y) dx dy \\
        & \quad + \iint \operatorname{Im} ( \bar u(y) \nabla u(y))\cdot \nabla_y \big(\operatorname{Im}(\bar u (x) \nabla u(x)) \cdot \nabla \rho(x-y)\big) dx dy \\
        &=: {\rm I a}+ {\rm Ib}.
    \end{aligned}
\end{equation*}
For the first term, notice that since $|\nabla \rho (x-y)| \lesssim 1$, we have, by the Cauchy-Schwarz inequality, then the mass law and the uniform bound on the energy Proposition \ref{propGWP},
\begin{equation*}
    \begin{aligned}
        \int_0^t |{\rm Ia}| d\tau &\leq \int_0^t \iint a(y) |u(x)| |\nabla u(x)| |\nabla \rho (x - y)| |u(y)|^{2\sigma_2+2} dx dy \, d\tau \\
        & \lesssim  \| u \|_{H^1}^2 \int_0^t \int a(y) |u(y)|^{2\sigma_2+2} dy \, d\tau \lesssim 1.
    \end{aligned}
\end{equation*}
For the second term, similarly to \cite[Proposition 2.5]{CoKeStTaTa04}, it follows from \eqref{eqRho} that
\begin{equation}
        {\rm Ib} \geq - \int\int \frac{|u(y)|^2|\pi_{(x-y)^\perp} \nabla u(x)|^2}{|x-y|} dx dy, \label{eq:Ic2}
\end{equation}
where
$$
\pi_{(x-y)^\perp} \xi := \xi - \frac{x-y}{|x-y|}\Big( \frac{x-y}{|x-y|}\cdot \xi\Big).
$$ 
Next, we pass to $\rm II$ in \eqref{eq:Bil3Terms} and we observe that, by \eqref{eqDNLS} it holds
\begin{equation*}
    \begin{aligned}
        {\rm II} &= \iint |u(y)|^2 |u(x)|^{2\sigma_1 + 2} \Delta \rho(x-y) dx dy \\ 
        & \quad + \iint |u(y)|^2 |\nabla u(x)|^2 \Delta \rho(x-y) dx dy \\ 
        & \quad + \iint |u(y)|^2 \operatorname{Re}(\bar u (x) \nabla u(x) ) \cdot \nabla_x \Delta \rho(x - y) dx dy  \\
        &\quad + \iint |u(y)|^2 V |u(x)|^{2\sigma_3+2}\Delta \rho(x-y) dxdy \\
        & =:  {\rm IIa} + {\rm IIb} + {\rm IIc} + {\rm IId}.
    \end{aligned}
\end{equation*}

By integration by parts, we get from \eqref{eqRho}
\begin{equation*}
    \begin{aligned}
        {\rm IIc} &= - \frac{1}{2} \iint |u(y)|^2 |u(x)|^2  \Delta^2 \rho(x-y) \,dx \,dy = 4\pi \int |u(x)|^4  \,dx.
    \end{aligned}
\end{equation*}
For the last term in \eqref{eq:Bil3Terms}, we have 
\begin{equation*}
    \begin{aligned}
        \frac{1}{2} {\rm III} & = \iint |u(y)|^2 a(x) \operatorname{Im}(\nabla \bar u(x) |u|^{2\sigma_2} u(x)) \cdot \nabla \rho(x-y) dx dy \\
        & \quad - \frac{1}{2\sigma_1 + 2} \iint |u(y)|^2 |u(x)|^{2\sigma_1 + 2} \Delta \rho(x-y) dx dy \\ 
        & \quad + \iint |u(y)|^2 \operatorname{Re} (\nabla u(x) \cdot \nabla_x( \nabla \bar u(x) \cdot \nabla \rho(x-y))\textcolor{black}{)} dx dy \\
        & \quad -\frac 12 \frac{1}{\sigma_3+1} {\rm IId} \\
        & \quad -\frac 12 \frac{1}{\sigma_3+1} \iint |u(y)|^2 (\nabla V(x)\cdot \nabla \rho(x-y)) |u(x)|^{2\sigma_3+2}dxdy \\
        & =: {\rm IIIa} + {\rm IIIb} + {\rm IIIc} -\frac 12 \frac{1}{\sigma_3+1} {\rm IId} + {\rm IIIe}.
    \end{aligned}
\end{equation*} 
For the first term in the above equation, it holds from \eqref{eqRho} as in \eqref{eqBadTerm}, then the local energy decay Proposition \ref{propGWP} and the mass law that
\begin{equation*}
    \left|\int_0^t{\rm IIIa} \, d\tau \right| \lesssim \| u \|_{L^\infty((0,t), L^2(\R^3))}^2 \left( \int_0^t \int  a|u|^{2\sigma_2}  | \nabla u |^2 dx \, d\tau +  \int_0^t \int  a |u|^{2\sigma_2+2} dx \, d\tau \right) \lesssim 1.
\end{equation*}
For the next term, we get
\begin{equation*}
    \begin{aligned}
        {\rm IIIc} & = \iint |u(y)|^2 \operatorname{Re}(D^2\rho (x - y)\nabla \bar u(x) \cdot \nabla u(x)) dx dy \\ 
        & \quad - \frac{1}{2} \iint |u(y)|^2 |\nabla u(x)|^2 \Delta \rho(x - y) dx dy \\
        & = \iint |u(y)|^2 \frac{|\pi_{(x-y)^\perp} \nabla u(x)|^2}{|x-y|} dx dy  - \frac{1}{2} \iint |u(y)|^2 |\nabla u(x)|^2 \Delta \rho(x - y) dx dy,
    \end{aligned}
\end{equation*}
where we used the fact that 
$$
\operatorname{Re}D^2\rho (x - y)\nabla \bar u(x) \cdot \nabla u(x) = \frac{|\pi_{(x-y)^\perp} \nabla u(x)|^2}{|x-y|}
$$ in the last line.

We finally turn to the control of 
the localized perturbation, i.e.\ ${\rm IIIe}$,
and the
trapping term, i.e.\ the negative part of $ {\rm IId}$. 

We begin with ${\rm IIIe}$.
Observe that, as $|\nabla \rho|=1$
$$
| {\rm IIIe} | \lesssim  \iint {|u(y)|^2} |u(x)|^{2\sigma_3+2} |\nabla V(x)|dxdy
= \Vert u \Vert^2_{L^2} \int |u(x)|^{2\sigma_3+2} |\nabla V(x)|dx,
$$
hence by the decay of the mass
$$
\int_0^t | {\rm IIIe} |d\tau \lesssim  \int_0^t\int|u(x)|^{2\sigma_3+2}|\nabla V(x)|\,dx\,d\tau.
$$

We use Young's product inequality, then the decay of $\nabla V$ in Assumption \ref{assumNV} combined with the local energy decay in Proposition \ref{propGWP}, to get
$$
\int_0^t | {\rm IIIe} | d\tau
\lesssim \int_0^t \int |\nabla V(x)|^{\frac{1}{\sigma_3+1}} |u(x)|^{2} dx \, d\tau +  \int_0^t \int |\nabla V(x)|^{\frac{\sigma_1+1}{\sigma_3+1}} |u(x)|^{2\sigma_1+2} dx \, d\tau \lesssim 1.
$$

We finally turn to the trapping term ${\rm IId }$. Observe that
$$
{\rm IId } \geq - \iint |u(y)|^2 V_-(x) |u(x)|^{2\sigma_3+2}\Delta \rho(x-y) dxdy = \textcolor{black}{-2} \iint V_-(x) \frac{|u(y)|^2 |u(x)|^{2\sigma_3+2}}{|x-y|} dxdy,
$$
where for fixed $x$, by the Cauchy-Schwarz inequality, then the Hardy inequality, and finally the uniform bound on the energy given by Proposition \ref{propGWP} and the decay of the mass
$$
\int \frac{|u(y)|^2}{|x-y|} dy \leq \Big\Vert \frac{u}{|\cdot-x|} \Big\Vert_{L^2} \Vert u \Vert_{L^2} \lesssim 2\Vert\nabla u\Vert_{L^2}\Vert u\Vert_{L^2}\lesssim1,
$$
hence
$$
{\rm IId } \gtrsim -\int V_-(x) |u(x)|^{2\sigma_3+2} dx.
$$
If $\sigma_3<\sigma_2$, we use Young's product inequality, then the decay of the trapping part  Assumption \ref{assumV} combined with the local energy decay Proposition \ref{propGWP}, to get
\begin{align*}
\int_0^t \int (V_-(x)) |u(x)|^{2\sigma_3+2} dx \, d\tau
&\lesssim \int_0^t \int (V_-(x))^{\frac{1}{\sigma_3+1}} |u(x)|^{2} dx \, d\tau \\
&+  \int_0^t \int (V_-(x))^{\frac{\sigma_2+1}{\sigma_3+1}} |u(x)|^{2\sigma_2+2} dx \, d\tau \\
&\lesssim 1 + \int_0^t\int (V_-(x))^{\frac{\sigma_2+1}{\sigma_3+1}} |u(x)|^{2\sigma_2+2} dx \, d\tau.
\end{align*}
Similarly, if $\sigma_3>\sigma_2$
\begin{align*}
\int_0^t \int (V_-(x)) |u(x)|^{2\sigma_3+2} dx \, d\tau
&\lesssim\int_0^t\int (V_-(x))^{\frac{\sigma_1+1}{\sigma_3+1}} |u(x)|^{2\sigma_1+2} dx \, d\tau \\
& + \int_0^t\int (V_-(x))^{\frac{\sigma_2+1}{\sigma_3+1}} |u(x)|^{2\sigma_2+2} dx \, d\tau \\
&\lesssim 1 + \int_0^t\int (V_-(x))^{\frac{\sigma_2+1}{\sigma_3+1}} |u(x)|^{2\sigma_2+2}dx d\tau,
\end{align*}\textcolor{black}{If $\sigma_3=\sigma_2$, Assumption \ref{assumControl} gives $V_-\leq c_0a$, so directly}
\[
\textcolor{black}{\int_0^t\int V_-|u|^{2\sigma_3+2}\,dx\,d\tau
\leq c_0\int_0^t\int a|u|^{2\sigma_2+2}\,dx\,d\tau
\leq\frac{c_0}{2}M[u(0)].}
\]
Thus, in all cases,
$$
\int_0^t \int (V_-(x)) |u(x)|^{2\sigma_3+2} dx \, d\tau \lesssim 1 + \int_0^t \int (V_-(x))^{\frac{\sigma_2+1}{\sigma_3+1}} |u(x)|^{2\sigma_2+2} dx \, d\tau \lesssim 1,
$$
using the control  Assumption \ref{assumControl} combined with the mass law \eqref{eq:MassEvol} to control the second term on the right-hand side.
Summing up, we have
$$
\int_0^t {\rm IId} \,  d\tau \gtrsim - 1.
$$ 

We can now conclude by integrating (\ref{eq:Bil3Terms}) and summing all the contributions above that
$$
4\pi \int_0^t \int |u|^4 dx \, d\tau + \frac{\sigma_1}{\sigma_1+1} \int_0^t \iint |u(y)|^2 |u(x)|^{2\sigma_1 + 2} \Delta \rho(x - y) dx dy  \, d\tau \lesssim 1,
$$
which ends the proof.
\end{proof}

\section{Scattering} \label{S:END}
In this section, we prove Theorem \ref{thmDefoSc1}. To handle the case where $a$ is asymptotically flat, we first prove the following lemma.

\begin{lem}\label{lem:flat_damping} Assume that the assumptions of Theorem \ref{th1} hold.
    Let \(0<\sigma_2\leq\sigma_1<2\), and  
    \(0 \leq a \in C^2(\R^3) \cap L^\infty(\R^3)\) so that $1-a \in C_c$. Then \textcolor{black}{there exists $C(u_0)>0$, such that}
    \begin{equation}\label{eqL2sigma2}
        \int_0^\infty \int_{\R^3} |u(x,\textcolor{black}{\tau})|^{2\sigma_2 + 2}\, dx\, d\tau\lesssim \textcolor{black}{C(u_0)}.
    \end{equation}
\end{lem}

\begin{proof}
    By the mass law \eqref{eq:MassEvol}, we have
    \begin{equation*}
        \int_0^\infty \int_{\R^3} |u(x,\textcolor{black}{\tau})|^{2\sigma_2 + 2}\, dx\, d\tau
        \lesssim 
        \| u_0 \|_{L^2(\R^3)}^2 
        + 
        \int_0^\infty \int_{\R^3} (1 - a(x))\, |u(x,\textcolor{black}{\tau})|^{2\sigma_2 + 2}\, dx\, d\tau.
    \end{equation*}
    We now show that the last term is bounded.  
    Choose $0\leq\phi\in C_c^2(\R^3)$ equal to one on $\operatorname{supp}(1-a)$.
    By interpolation we have for some \( \theta \in [0,1]\),
    \begin{equation*}
        \| \phi u \|_{L^{2\sigma_2+2}((0,\infty), L^{2\sigma_2+2}(\R^3))}
        \lesssim   
        \| \phi u \|_{L^{2\sigma_1+2}((0,\infty), L^{2\sigma_1+2}(\R^3))}^{\theta }
        \,
        \| \phi u \|_{L^{2}((0,\infty), L^2(\R^3))}^{1-\theta } \lesssim 1,
    \end{equation*}
    where we controlled both factors by the local energy decay \eqref{eqLocEnDec}.  
    Since $|1-a|\leq\|1-a\|_{L^\infty}\phi^{2\sigma_2+2}$, the absolute value of the remaining integral is bounded. This proves \eqref{eqL2sigma2}.
\end{proof}

\begin{proof}[Proof of Theorem \ref{thmDefoSc1}]
$\;$

\textcolor{black}{\emph{$\triangleright$ Assertion (1).}
We denote the entire source term as}
\[
\textcolor{black}{F(u)=|u|^{2\sigma_1}u+V|u|^{2\sigma_3}u-ia|u|^{2\sigma_2}u,
}
\]
and let $\textcolor{black}{B = \sup_{t\geq 0} \| u(t) \|_{H^1} <\infty}$.
\textcolor{black}{The goal is the global Strichartz bound}
\begin{equation}\label{eqZ1}
\textcolor{black}{Z:=\sup_{(p,q)\text{ admissible}}
\|u\|_{L^p((0,\infty),W^{1,q})}\leq C(u_0).}
\end{equation}
Using the global \(L^4_{t,x}\) control \eqref{eq:GlStrEst} , we decompose 
\([0,+\infty)\) into a finite collection of intervals \(J_1,\dots,J_K\) such that
for each \(i=1,\dots,K\),
\begin{equation} \label{eq:smallL4L4}
        \| u \|_{L^4(J_i, L^4(\R^3))} \le \varepsilon,
\end{equation}
where \(\varepsilon=\varepsilon(\|u_0\|_{H^1(\R^3)})>0\) will be chosen small.
\textcolor{black}{For $t_j<t<t_{j+1}$, set $I_{j,t}=(t_j,t)$ and define
$$Z_j(t)=\|u\|_{L^\infty(I_{j,t},H^1)}+\|u\|_{L^2(I_{j,t},W^{1,6})}.$$
Interpolation controls every Strichartz admissible norm $\Vert \cdot \Vert_{L^p(I_{j,t}, W^{1,q})}$ on $I_{j,t}$ by $C Z_j(t)$.
Restarting Duhamel's formula at $t_j$ gives}
\[
\textcolor{black}{u(t)=e^{i(t-t_j)\Delta}u(t_j)
-i\int_{t_j}^t e^{i(t-s)\Delta}F(u(s))\,ds.}
\]
\textcolor{black}{All norms in the next estimates are on $I=I_{j,t}$.
By Proposition \ref{prop:strich}, each nonlinear term is estimated using either
$L^{10/7}(I,W^{1,10/7})$ or $L^2(I,W^{1,6/5})$, and the homogeneous term is bounded by $CB$.}

\textcolor{black}{First let $2/3<\sigma<2$ and $\psi\in W^{1,\infty}$.
The chain rule and H\"older's inequality give}
\begin{equation}\label{eqZControl1}
\textcolor{black}{\|\psi |u|^{2\sigma}u\|_{L^{10/7}(I,W^{1,10/7})}
\lesssim \|\psi\|_{W^{1,\infty}}
\|u\|_{L^{10/3}(I,W^{1,10/3})}\|u\|_{L^{5\sigma}(I\times\R^3)}^{2\sigma}.}
\end{equation}
\textcolor{black}{For $2/3<\sigma\leq4/5$, interpolate $L^{5\sigma}_{t,x}$
between $L^4_{t,x}$ and $L^{10/3}_{t,x}$; for $4/5<\sigma<2$, interpolate it
between $L^4_{t,x}$ and $L^{10}_{t,x}$. The latter norm satisfies}
\begin{equation}\label{eqL10L10}
\textcolor{black}{\|u\|_{L^{10}(I\times\R^3)}
\lesssim\|u\|_{L^{10}(I,W^{1,30/13})}\lesssim Z_j(t).}
\end{equation}
\textcolor{black}{Thus, by \eqref{eqL10L10}, \eqref{eqZControl1}, and \eqref{eq:smallL4L4}}
\begin{equation}\label{eqZControl}
\textcolor{black}{\begin{gathered}
\|\psi |u|^{2\sigma}u\|_{L^{10/7}(I,W^{1,10/7})}
\leq C_\psi\varepsilon^{\alpha_\sigma}Z_j(t)^{1+\delta_\sigma},\\
(\alpha_\sigma,\delta_\sigma)=
\begin{cases}
(12\sigma-8,\,8-10\sigma),&2/3<\sigma\leq4/5,\\
((8-4\sigma)/3,\,(10\sigma-8)/3),&4/5<\sigma<2.
\end{cases}
\end{gathered}}
\end{equation}
\textcolor{black}{Here $\alpha_\sigma>0$ and $\delta_\sigma\geq0$;
in particular $\delta_{4/5}=0$, which causes no difficulty.}

\textcolor{black}{Next suppose $0<\sigma\leq2/3$ and $(\psi,\sigma)$ is
intercritical in Definition \ref{assumLoc}. Set}
\[
\textcolor{black}{p_\sigma=\frac{6}{4-5\sigma},\qquad
r_\sigma=\frac{2}{1-\sigma},\qquad s_\sigma=\frac{6}{1+2\sigma}.}
\]
\textcolor{black}{The pair $(r_\sigma,s_\sigma)$ is admissible,
$1/r_\sigma+\sigma/2=1/2$, and
$1/p_\sigma+1/s_\sigma+\sigma/2=5/6$. Therefore}
\begin{equation}\label{eqLocalizedNonlinearity}
\textcolor{black}{\begin{aligned}
\|\psi|u|^{2\sigma}u\|_{L^2(I,W^{1,6/5})}
&\lesssim\|\psi\|_{W^{1,p_\sigma}}
\|u\|_{L^{r_\sigma}(I,W^{1,s_\sigma})}
\|u\|_{L^4(I\times\R^3)}^{2\sigma}\\
&\leq C_\psi\varepsilon^{2\sigma}Z_j(t).
\end{aligned}}
\end{equation}

\textcolor{black}{Applying these estimates to $(1,\sigma_1)$, $(a,\sigma_2)$,
and $(V,\sigma_3)$ gives}
\begin{equation}\label{eqZControl0}
\textcolor{black}{Z_j(t)\leq CB+
\sum_\ell C_\ell\varepsilon^{\alpha_\ell} Z_j(t)^{1+\delta_\ell},
\qquad\alpha_\ell>0,\quad\delta_\ell\geq0.}
\end{equation}
\textcolor{black}{If $B=0$ there is nothing to prove. Otherwise fix $A\geq CB+B$,
and choose $\varepsilon$ small enough that
$$\sum_\ell C_\ell\varepsilon^{\alpha_\ell}(2A)^{1+\delta_\ell}\leq A/2.$$
A local continuity argument gives $Z_j(t)\leq2A$ on every $J_j$, with the
same $A$ and $\varepsilon$, because the initial value on each interval has
$H^1$ norm at most $B$. Summing the squared $L^2W^{1,6}$ norms over the finite
partition gives a global $L^2W^{1,6}$ bound. Together with $L^\infty H^1$ and
interpolation, this proves \eqref{eqZ1}.}

\textcolor{black}{We finally show the scattering, by defining}
\begin{equation}\label{eqScState}
\textcolor{black}{\begin{aligned}
u_+:={}&u_0-i\int_0^\infty e^{-is\Delta}(|u|^{2\sigma_1}u)(s)\,ds\\
&-\int_0^\infty e^{-is\Delta}(a|u|^{2\sigma_2}u)(s)\,ds
-i\int_0^\infty e^{-is\Delta}(V|u|^{2\sigma_3}u)(s)\,ds.
\end{aligned}}
\end{equation}
\textcolor{black}{To justify these integrals, first integrate over finite intervals $(t,T')$ and apply the dual Strichartz estimate. The bounds (\ref{eqNonLinSc}) below %uniformly in $T'>t$, 
 imply the Cauchy criterion in $H^1$ and justify passing to $T'=\infty$. For each
$(\psi,\sigma)\in\{(1,\sigma_1),(a,\sigma_2),(V,\sigma_3)\}$, the dual
Strichartz estimate and the preceding nonlinear estimates, now on $(t,\infty)$, yield}
\begin{equation}\label{eqNonLinSc}
\textcolor{black}{\begin{aligned}
&\left\|\int_t^\infty e^{-is\Delta}(\psi|u|^{2\sigma}u)(s)\,ds\right\|_{H^1}\\
&\qquad\lesssim
\begin{cases}
Z^{1+\delta_\sigma}\|u\|_{L^4((t,\infty)\times\R^3)}^{\alpha_\sigma},&\sigma>2/3,\\
Z\|u\|_{L^4((t,\infty)\times\R^3)}^{2\sigma},&0<\sigma\leq2/3\text{ intercritical}.
\end{cases}
\end{aligned}}
\end{equation}
\textcolor{black}{The right-hand sides of the above tend to zero by \eqref{eq:GlStrEst} and \eqref{eqZ1}. Hence by (\ref{eqScState})}
\begin{equation}\label{eq:scat_crit_I1NT}
\textcolor{black}{\|u(t)-e^{it\Delta}u_+\|_{H^1}\longrightarrow0.}
\end{equation}

\textcolor{black}{\emph{$\triangleright$ Assertion (2).}
We begin with the case where $\sigma_2=2/3$. Then, Lemma \ref{lem:flat_damping} directly gives $u\in L^{10/3}_{t,x}$.
Refine the finite partition so that both \eqref{eq:smallL4L4} and
$\|u\|_{L^{10/3}(J_j\times\R^3)}\leq\varepsilon$ hold. For any term with
$\sigma=2/3$ and $\psi\in W^{1,\infty}$, we have}
\begin{equation}\label{eqCriticalNonlinearity}
\textcolor{black}{\begin{aligned}
\|\psi|u|^{4/3}u\|_{L^{10/7}(I,W^{1,10/7})}
&\lesssim\|\psi\|_{W^{1,\infty}}
\|u\|_{L^{10/3}(I,W^{1,10/3})}\|u\|_{L^{10/3}(I\times\R^3)}^{4/3}\\
&\leq C_\psi\varepsilon^{4/3}Z_j(t).
\end{aligned}}
\end{equation}
\textcolor{black}{This additional linear term is absorbed in exactly the same
local bootstrap. For scattering, its tail is bounded by
$C_\psi Z\|u\|_{L^{10/3}((t,\infty)\times\R^3)}^{4/3}\to0$.}

\textcolor{black}{If instead $\sigma_2>2/3$ and $(V,\sigma_3)$ is intercritical, as $\sigma_1\geq \sigma_2$, we necessarily have $\sigma_1>2/3$. The exact same arguments as in assertion (1) (in the case where $\sigma_2>2/3$) apply.}

\end{proof}

\bibliographystyle{abbrv}
\bibliography{biblio}

\end{document}